\documentclass[11pt, reqno]{amsart}

\usepackage[table]{xcolor}
\usepackage{amsmath, amssymb, mathtools, ytableau, tikz, multirow, colortbl, hhline, graphicx, subcaption, diagbox, booktabs}

\newtheorem{theorem}{Theorem}[section]
\newtheorem{lemma}[theorem]{Lemma}
\newtheorem{corollary}[theorem]{Corollary}

\theoremstyle{remark}
\newtheorem*{remark}{Remark}

\numberwithin{equation}{section}

\makeatletter
\@namedef{subjclassname@2020}{%
  \textup{2020} Mathematics Subject Classification}
\makeatother

\newcommand{\Log}{\operatorname{Log}}
\renewcommand{\Re}{{\rm Re}}

\begin{document}


\title[Inequalities and asymptotics for hook lengths in $\ell$-regular and $\ell$-distinct partitions]{Inequalities and asymptotics for hook lengths in $\ell$-regular partitions and $\ell$-distinct partitions}

\author{Eunmi Kim}
\address{Institute of Mathematical Sciences, Ewha Womans University,  52 Ewhayeodae-gil, Seodaemun-gu, Seoul 03760, Republic of Korea}
\email{ eunmi.kim67@gmail.com; ekim67@ewha.ac.kr}


\begin{abstract} 
In this article, we study hook lengths in $\ell$-regular partitions and $\ell$-distinct partitions. More precisely, we establish hook length inequalities between $\ell$-regular partitions and $\ell$-distinct partitions for hook lengths $2$ and $3$, by deriving asymptotic formulas for the total number of hooks of length $t$ in both partition classes, for $t = 1, 2, 3$. From these asymptotics, we show that the ratio of the total number of hooks of length $t$ in $\ell$-regular partitions to those in $\ell$-distinct partitions tends to a constant that depends on $\ell$ and $t$. We also provide hook length inequalities within $\ell$-regular partitions and within $\ell$-distinct partitions.
\end{abstract}


\maketitle

\section{introduction}
A {\it partition} $\lambda=(\lambda_1, \lambda_2, \dots, \lambda_k)$ of a positive integer $n$ is a non-increasing sequence of positive integers $\lambda_1\ge\lambda_2\ge\dots\ge\lambda_k$ such that the {\it parts} $\lambda_j$ sum up to $n$. 
A partition can be represented as a {\it Young diagram}, which is a left-justified array of square boxes. For each box $v$ in the Young diagram of a partition $\lambda$, the {\it arm length}  (resp. the {\it coarm length}, the {\it leg length}) of $v$, denoted by $a_\lambda(v)$ (resp. $ca_\lambda(v)$, $l_\lambda(v)$), is defined as the number of boxes to the right of (resp. to the left of, below) $v$ in the diagram of $\lambda$. See Figure \ref{fig:arm_leg}. The {\it hook length} of $v$ is $a_\lambda(v)+l_\lambda(v)+1$. In Figure \ref{fig:hook_length},  the Young diagram of the partition $\lambda=(5, 4, 2, 1)$ with hook lengths is illustrated.

\begin{figure}[ht]
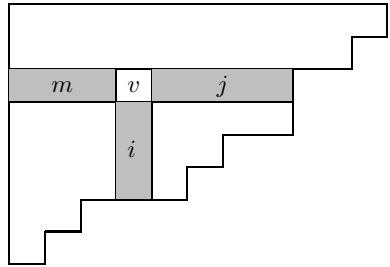
	
	\scalebox{0.85}{
	\begin{tabular}{|lllllllllll}
	\hline
         	&   &   &   &   &   &   &   &   &   & \multicolumn{1}{l|}{} \\ \cline{11-11} 
         	&   &   &   &   &   &   &   &   & \multicolumn{1}{l|}{} &  \\ \cline{1-10}\hhline{--~}
		\multicolumn{3}{|c|}{\cellcolor{lightgray}$m$}  & \multicolumn{1}{l|}{$v$}  & \multicolumn{4}{c|}{\cellcolor{lightgray}$j$}  &   &   &  \\ \cline{1-8}\hhline{--~}
         	&   & \multicolumn{1}{l|}{} & \multicolumn{1}{l|}{\cellcolor{lightgray}}  &   &   &  & \multicolumn{1}{l|}{} &   &   &   \\ \cline{7-8}
         	&   & \multicolumn{1}{l|}{} & \multicolumn{1}{l|}{\cellcolor{lightgray}}  &   & \multicolumn{1}{l|}{} &  &   &   &   &    \\ \cline{6-6}
         	&   & \multicolumn{1}{l|}{} & \multicolumn{1}{l|}{\multirow{-3}{*}{\cellcolor{lightgray}$i$}} & \multicolumn{1}{l|}{} &   &   &   &   &   &  \\ \cline{3-5}
         	& \multicolumn{1}{l|}{} &  &   &   &   &   &   &   &   &   \\ \cline{2-2}
		\multicolumn{1}{|l|}{\hphantom{$v$}  } &  \hphantom{$v$}  & \hphantom{$v$} &  \hphantom{$v$} & \hphantom{$v$}  &  \hphantom{$v$} & \hphantom{$v$}  &  \hphantom{$v$}  & \hphantom{$v$}   &   & \hphantom{$v$} \\ \cline{1-1}
	\end{tabular}}
	\caption{Arm, coarm, and leg lengths of $v$: $a_\lambda(v)=j$, $ca_\lambda(v)=m$, and $l_\lambda(v)=i$}
	\label{fig:arm_leg}
\end{figure}

\begin{figure}[ht]
	\begin{ytableau} 
		8&6&4&3 & 1\\
		6&4&2&1\\
		3&1\\
		1
	\end{ytableau}
	\caption{The Young diagram of the partition ${(5,4, 2, 1)}$ with hook lengths}
	\label{fig:hook_length}
\end{figure}

Hook lengths in integer partitions play an important role in the representation theory of symmetric groups. Also, the Nekrasov-Okounkov formula \cite{NO} connects hook lengths in partitions with modular forms and $q$-series. Han \cite{H} generalizes it using a combinatorial bijection \cite{GKS} involving partitions, $t$-cores, and $t$-quotients. 
These formulas lead to an extensive study on hook lengths in partitions (see e.g. \cite{AS, BBCFW,  BOW, CDH, COS, GOT}).

A {\it $\ell$-regular partition} of a positive integer $n$ is a partition of $n$ in which no part is divisible by $\ell$ and 
a {\it $\ell$-distinct partition} of a positive integer $n$ is a partition of $n$ in which parts appear fewer than $\ell$ times. Glaisher bijectively proved that the number of $\ell$-regular partitions of $n$, denoted by $b_\ell(n)$, is equal to the number of $\ell$-distinct partitions of $n$, denoted by $d_\ell(n)$. Their generating functions are
\[
	\sum_{n \geq 0} b_\ell (n)q^n=\prod_{\substack{n \geq 1\\ \ell \nmid n}}\frac{1}{1-q^n} = \frac{(q^\ell;q^\ell)_\infty}{(q;q)_\infty} = \prod_{n \geq 1} \left( 1+ q^n + \cdots + q^{(\ell-1)n} \right) = \sum_{n \geq 0} d_\ell (n)q^n,
\]
where, for $a \in \mathbb{C}$, $n \in \mathbb{N}_0\cup\{\infty\}$, we let $(a;q)_n :=\!\prod_{j=0}^{n-1}(1-aq^j)$. Note that a $2$-regular partition is a partition into odd parts and that a $2$-distinct partition is a partition into distinct parts.

For integers $\ell \geq 2$ and $t \geq 1$, let $b_{\ell,t}(n)$ (resp. $d_{\ell,t}(n)$) be the total number of hooks of length $t$ in all the $\ell$-regular partitions (resp. $\ell$-distinct partitions) of $n$. 
Andrews \cite[Theorem 2]{A2} proved for all $n \geq 0$ that the difference $d_{2,1}(n) - b_{2,1}(n)$ is equal to the number of partitions of $n$ such that there is exactly one part occurring three times while all other parts occur only once, which implies that $d_{2,1}(n) \geq b_{2,1}(n)$ for all $n \geq 0$. From Glaisher's bijection, we have that for each $\ell \geq 2$,
\[
	b_\ell(n) = \sum_{t \geq 1} b_{\ell,t}(n) = \sum_{t \geq 1} d_{\ell,t}(n) = d_\ell(n)
\]
for all $n \geq 0$. Thus, for each $n \geq 0$, $b_{2,t}(n) \geq d_{2,t}(n)$ should hold for some $t \geq 2$. Ballantine, Burson, Craig, Folsom, and Wen \cite{BBCFW} conjectured that, for $t \geq 2$, there exists $N_t$ such that $b_{2,t}(n) \geq d_{2,t}(n)$ for all $n > N_t$. Craig, Dawsey, and Han \cite[Theorem 1.2]{CDH} confirmed their conjecture by giving the asymptotics of $b_{2,t}(n)$ and $d_{2,t}(n)$ as $n \to \infty$.

Li and Wang \cite[Theorem 1.6]{LW} proved that for $n \geq 0$ and $\ell \geq 2$, the difference $d_{\ell,1}(n) - b_{\ell,1}(n)$ equals the number of partitions of $n$ in which exactly one part appears more than $\ell$ times but fewer than $2\ell$ times, while all other parts occur fewer than $\ell$ times. Thus, for $\ell \geq 2$, we have $d_{\ell,1}(n) \geq b_{\ell,1}(n)$ for all $n \geq 0$. By Glaisher's bijection, we also expect for each $\ell \geq 3$ that the inequality $b_{\ell,t}(n) \geq d_{\ell,t}(n)$ holds for some $t\geq2$.

In this article, for each $\ell \geq 2$, we establish inequalities between $b_{\ell,t}(n)$ and $d_{\ell,t}(n)$ for sufficiently large integers $n$ when $t=2$ and $t=3$. Table \ref{tab:bias} illustrates inequalities between $b_{\ell,t}(n)$ and $d_{\ell,t}(n)$ for $n \gg 0$. Moreover, when $t=1, 2, 3$, we prove that for $\ell \geq 2$, the ratio $\frac{d_{\ell,t}(n)}{b_{\ell,t}(n)}$ tends to a constant $r_{\ell, t}>0$ as $n \to \infty$. We also show that $r_{\ell,t} \to 1$ as $\ell \to \infty$, which implies that although there are inequalities between $b_{\ell,t}(n)$ and $d_{\ell,t}(n)$, their ratio approaches $1$ as $\ell \to \infty$.

\begin{table}
	\begin{tabular}{c|ccccccc}
		\diagbox[width=23pt, height=23pt]{$\ell$}{$t$} & $1$ & $2$ & $3$  & $4$  & $5$ & $6$ & $\cdots$  \\ \hline
		$2~$ & \cellcolor{lightgray!60} $d~$ & \cellcolor{gray!80} $b~$ & \cellcolor{gray!80} $b~$ & \cellcolor{gray!80} $b~$ & \cellcolor{gray!80} $b~$ & \cellcolor{gray!80} $b~$ & \cellcolor{gray!80}$\cdots$ \\
		$3~$ & \cellcolor{lightgray!50} $d~$ & \cellcolor{gray!80} $b~$ & \cellcolor{gray!80} $b~$ & $b?$ & $b?$ & $b?$  \\ 
		$4~$ & \cellcolor{lightgray!50} $d~$ & \cellcolor{lightgray!50} $d~$ & \cellcolor{gray!80} $b~$ & $b?$  &  \\ 
		$5~$ & \cellcolor{lightgray!50} $d~$  & \cellcolor{lightgray!50}  $d~$ & \cellcolor{lightgray!50} $d~$ &  & $b?$ \\
		$6~$ & \cellcolor{lightgray!50} $d~$  & \cellcolor{lightgray!50}  $d~$ & \cellcolor{lightgray!50} $d~$ &  &  & $b?$ \\
 		$\vdots~$ & \cellcolor{lightgray!50} $\vdots~$ & \cellcolor{lightgray!50} $\vdots~$  & \cellcolor{lightgray!50} $\vdots~$ &  &  &  &                         
	\end{tabular}
	\bigskip
	\caption{For each $\ell \geq 2$ and $t \geq 1$, the larger of $b_{\ell,t}(n)$ and $d_{\ell,t}(n)$ for $n \gg 0$ is indicated as either `$b$' or `$d$'. The first row is confirmed by \cite[Theorem 1.2]{CDH}, and the first column is confirmed by \cite[Theorem 1.6]{LW}. The second and third columns are verified by Theorem \ref{thm:bias_b_d}. Additionally, `$b?$' or `$d?$' indicates the conjecture based on numerical data for $n$ up to $100$.}
	\label{tab:bias}
\end{table}

\begin{theorem}\label{thm:bias_b_d}  
	Let $\ell \geq 2$ be an integer.  
	\begin{enumerate} 
		\item For sufficiently large integers $n$, we have that
		\begin{align*}
			b_{\ell,2}(n) &\geq d_{\ell,2}(n) \quad \text{for }~ 2 \leq \ell \leq 3, & b_{\ell,2}(n) \leq d_{\ell,2}(n) \quad \text{for }~ \ell \geq 4,\\
			b_{\ell,3}(n) &\geq d_{\ell,3}(n) \quad \text{for }~ 2 \leq \ell \leq 4, &b_{\ell,3}(n) \leq d_{\ell,3}(n) \quad \text{for }~ \ell \geq 5.
		\end{align*}
		\item Furthermore, for each $t=1, 2, 3$, there exists a constant $r_{\ell,t}>0$ such that as $n \to \infty$,
		\[
			\frac{d_{\ell,t}(n)}{b_{\ell,t}(n)} \to r_{\ell,t},
		\]
		where we have that  $r_{\ell,t} \to 1$ as $\ell \to \infty$.
	\end{enumerate}
\end{theorem}

Theorem \ref{thm:bias_b_d} follows from the asymptotic formulas of $b_{\ell,t}(n)$ and $d_{\ell,t}(n)$ for $t=1, 2, 3$.

\begin{theorem}\label{thm:b_asymp}
	Let $\ell \geq 2$ be a fixed integer. As $n \to \infty$,  
	\begin{align*}
		b_{\ell,1}(n) &= \frac1{\pi} \left( 1- \frac1{\ell} \right)  \left( \frac{3}{8\ell(\ell-1)} \right)^{\frac14} n^{-\frac14} e^{\pi \sqrt{\frac{2n}{3}\left(1-\frac1\ell\right)}}\left( 1+  O\left(\frac{1}{\sqrt n}\right) \right),\\
		b_{\ell,2}(n) &=  \frac1{\pi} \left( 1- \frac1{2\ell} \right)  \left( \frac{3}{8\ell(\ell-1)} \right)^{\frac14} n^{-\frac14} e^{\pi \sqrt{\frac{2n}{3} \left(1- \frac1\ell\right)}}  \\
		&\qquad \times  \left[ 1 +\left( \frac{\pi(\ell-1)}{24} \sqrt{\frac{\ell-1}{6\ell}} - \frac{2\pi}{2\ell-1} \sqrt{\frac{\ell(\ell-1)}{6}} + \frac1{16\pi} \sqrt{\frac{6\ell}{\ell-1}} \right) \frac1{\sqrt{n}} + \mathcal{O}\left( \frac1{n} \right) \right],\\
		b_{\ell,3}(n)  &=   \frac1{\pi} \left( 1- \frac1{2\ell} \right)  \left( \frac{3}{8\ell(\ell-1)} \right)^{\frac14} n^{-\frac14} e^{\pi \sqrt{\frac{2n}{3} \left(1- \frac1\ell\right)}} \\
		&\qquad \times  \left[ 1 +\left( \frac{\pi(\ell-1)}{24} \sqrt{\frac{\ell-1}{6\ell}} - \frac{3\pi}{2\ell-1} \sqrt{\frac{\ell(\ell-1)}{6}} + \frac1{16\pi} \sqrt{\frac{6\ell}{\ell-1}} \right) \frac1{\sqrt{n}} + \mathcal{O}\left( \frac1{n} \right) \right].
	\end{align*}
\end{theorem}

\begin{theorem}\label{thm:d_asymp}
	Let $\ell \geq 2$ be  a fixed integer. For $t=1, 2, 3$, we have that as $n \to \infty$,  
	\[
		d_{\ell,t}(n)= \frac{\beta_t(\ell)}{\pi} \left( \frac{3}{8 \ell(\ell-1)}\right)^{\frac14} n^{-\frac14} e^{\pi \sqrt{\frac{2n}{3}\left( 1- \frac1{\ell}\right)}} \left( 1+  O\left(\frac{1}{\sqrt n}\right) \right),
	\]
	where 
	\begin{align*}
		\beta_1(\ell) &:= \frac1\ell \left[ -\gamma - \psi\left(\frac1\ell \right)\right],\\
		\beta_2(\ell) &:= \frac1\ell \left[ 1- \gamma - \left(1-\frac2\ell\right) \psi\left( \frac1\ell \right) -\frac2\ell \psi\left( \frac2\ell \right) \right], \\
		\beta_3(\ell) &:=  \frac1\ell \left[\frac32 -\gamma - \left( 1- \frac{3}{2\ell}\right) \left( 1-\frac1\ell \right)\psi\left( \frac1\ell \right) - \frac{1}{\ell} \left( 1-\frac6\ell \right)\psi\left( \frac2\ell \right) -  \frac{3}{2\ell} \left( 1+\frac3\ell \right)\psi\left( \frac3\ell \right) \right].
	\end{align*} 
	Here, $\gamma$ is the Euler-Mascheroni constant, and $\psi(a):=\frac{\Gamma'(a)}{\Gamma(a)}$ is the digamma function.
\end{theorem}

Theorems \ref{thm:bias_b_d}, \ref{thm:b_asymp}, and \ref{thm:d_asymp} are illustrated in Table \ref{tab:comparison}. We denote the main term of the asymptotic formula for $b_{\ell,t}(n)$ in Theorem \ref{thm:b_asymp} (resp. $d_{\ell,t}(n)$ in Theorem \ref{thm:d_asymp}) by $b^{[a]}_{\ell,t}(n)$ (resp. $d^{[a]}_{\ell,t}(n)$). We note that the inequalities in Theorem \ref{thm:bias_b_d}  seem to hold for all $n$ with a few exceptions at the beginning. 

\begin{table}[!h]
\begin{tabular}{cc|ccc|ccc|ccc}
	\toprule
	$t=2$&&& $\ell=3$ && & $\ell=4$& & & $\ell=5$& \\
	&$n$ &  $\frac{d_{3,2}(n)}{b_{3,2}(n)}$ & $\frac{b_{3,2}(n)}{b^{[a]}_{3,2}(n)}$ &  $\frac{d_{3,2}(n)}{d^{[a]}_{3,2}(n)}$ &  $\frac{d_{4,2}(n)}{b_{4,2}(n)}$ & $\frac{b_{4,2}(n)}{b^{[a]}_{4,2}(n)}$ &  $\frac{d_{4,2}(n)}{d^{[a]}_{4,2}(n)}$&  $\frac{d_{5,2}(n)}{b_{5,2}(n)}$& $\frac{b_{5,2}(n)}{b^{[a]}_{5,2}(n)}$ &  $\frac{d_{5,2}(n)}{d^{[a]}_{5,2}(n)}$\\
	\midrule
	&$100$ & $0.9621$ & $0.8989$ & $0.9218$ & $1.0247$ & $0.9028$ & $0.9204$ & $1.0416$ & $0.9078$ & $0.9209$  \\
	&$500$ & $ 0.9497$ & $ 0.9525$ & $0.9641$ & $1.0151$ & $0.9541$ & $0.9635$ & $1.0350$  & $0.9563$ & $0.9639$ \\
	&$1000$ & $0.9465$ & $0.9660$ & $0.9744$ & $1.0124$ & $0.9671$  & $0.9741$ & $1.0329$ & $0.9686$ & $0.9743$  \\ 
	&$5000$ & $0.9420$ & $0.9845$ & $0.9885$ & $1.0086$ & $0.9850$ & $0.9883$ & $1.0297$  & $0.9857$ & $0.9884$  \\
	\midrule
	&$r_{\ell,2}$ & $0.9382$ & & &  $1.0052$ &  & & $1.0268$ & & \\
	\bottomrule
\end{tabular}\\
\bigskip
\begin{tabular}{cc|ccc|ccc|ccc}
	\toprule
	$t=3$&&& $\ell=3$ && & $\ell=4$& & & $\ell=5$& \\
	& $n$ & $\frac{d_{3,3}(n)}{b_{3,3}(n)}$ & $\frac{b_{3,3}(n)}{b^{[a]}_{3,3}(n)}$ &  $\frac{d_{3,3}(n)}{d^{[a]}_{3,3}(n)}$ &  $\frac{d_{4,3}(n)}{b_{4,3}(n)}$ &  $\frac{b_{4,3}(n)}{b^{[a]}_{4,3}(n)}$ &  $\frac{d_{4,3}(n)}{d^{[a]}_{4,3}(n)}$ &  $\frac{d_{5,3}(n)}{b_{5,3}(n)}$ & $\frac{b_{5,3}(n)}{b^{[a]}_{5,3}(n)}$ &  $\frac{d_{5,3}(n)}{d^{[a]}_{5,3}(n)}$\\
	\midrule
	&$100$ & $0.8648$ & $0.8428$ & $0.8658$ & $0.9861$ & $0.8460$ & $0.8661$ & $1.0236$ & $0.8505$ & $0.8664$  \\
	&$500$ & $0.8529$ & $0.9257$ & $0.9379$  & $0.9747$ & $0.9271$ & $0.9381$ & $1.0149$ & $0.9291$ & $0.9384$ \\
	&$1000$ & $0.8498$ & $0.9468$ & $0.9557$ & $0.9716$ & $0.9477$ & $0.9559$ & $1.0122$ & $0.9491$ & $0.9561$ \\	 
	&$5000$ &$0.8455$ &  $0.9758$ & $0.9800$ & $0.9671$ & $0.9762$ & $0.9801$ & $1.0083$ & $0.9768$ & $0.9802$ \\	 
	\midrule
	&$r_{\ell,3}$ & $0.8418$ & & &  $0.9633$& & &  $1.0048$ & &\\
	\bottomrule
\end{tabular}
\medskip
\caption{Comparison of $b_{\ell,t}(n)$ and $d_{\ell,t}(n)$ with their asymptotic values $b_{\ell,t}^{[a]}(n)$ and $d_{\ell,t}^{[a]}(n)$  (values rounded to four decimal places). Note that $\frac{d_{\ell,t}(n)}{b_{\ell,t}(n)} \to r_{\ell,t}$ as $n \to \infty$.}
\label{tab:comparison}
\end{table}

As corollaries of Theorem \ref{thm:b_asymp}, we find the following inequalities on $b_{\ell,t}(n)$ for $n \gg 0$.
\begin{corollary}\label{cor:reg1}
	Let $\ell \geq 2$ be an integer. For sufficiently large integers $n$,
	\[
		b_{\ell,2}(n) \geq b_{\ell,1}(n)  \qquad \text{and} \qquad b_{\ell,2}(n) \geq b_{\ell,3}(n). 
	\]
\end{corollary}

\begin{corollary}\label{cor:reg2}
	Let $\ell \geq 2$ be an integer. For each $t=1, 2, 3$, there exist some positive integers $N_{\ell}$ such that for all $n > N_\ell$, 
	\[
		b_{\ell+1,t}(n) \geq b_{\ell,t}(n).
	\]
	In particular, for each $t=1, 2, 3$, as $n \to \infty$,
	\[
		\frac{b_{\ell+1,t}(n)}{b_{\ell,t}(n)} \to \infty.
	\]
\end{corollary}

Singh and Barman \cite{SB1}\footnote{In Theorem 1.3 of this article, the generating function of $b_{4,2}(n)$ is stated incorrectly. See Section 3.1.} proved that $b_{2,2}(n) \geq b_{2,1}(n)$ for all $n > 4$ and $b_{2,2}(n) \geq b_{2,3}(n)$ for all $n \geq 0$, and conjectured  that $b_{3,2}(n) \geq b_{3,1}(n)$ for $n \geq 28$ (see Theorems 1.4, 1.5, and Conjecture 6.1). Corollary \ref{cor:reg1} generalizes their theorems and proves the conjecture asymptotically. In \cite{SB2}, it is proved that $b_{\ell+1,1}(n) \geq b_{\ell,1}(n)$ for all $n \geq 0$, and it is conjectured that for a fixed $\ell \geq 3$, $b_{\ell+1,2}(n) \geq b_{\ell,2}(n)$ holds for all $n \geq 0$ (see Theorem 1.1 and Conjecture 4.1). Corollary \ref{cor:reg2} confirms the conjecture for $n\gg 0$.

\begin{remark}
	In \cite[Conjecture 1.6]{SB1}, the authors conjectured that for every integer $t \geq 3$, $b_{2,t}(n) \geq b_{2,t+1}(n)$ for all $n \geq 0$ and $n \neq t+1$. This conjecture is true only if $t$ is even. In fact, for $t \geq 1$, we can verify that $b_{2, 2t}(n) \geq b_{2, 2t+2}(n)$ and $b_{2,2t}(n) \geq b_{2,2t+1}(n) \geq b_{2,2t-1}(n)$ for $n \gg 0$ by using the asymptotic formula of $b_{2,t}(n)$ in \cite[Theorems 1.4 and 4.6]{CDH}. 
\end{remark}

As a corollary of Theorem \ref{thm:d_asymp}, we also obtain the inequalities of $d_{\ell,t}(n)$ for $n \gg 0$.
\begin{corollary}\label{cor:dis1}
	Let $\ell \geq 2$ be an integer. For sufficiently large integers $n$,
	\[
		d_{\ell,1}(n) \geq d_{\ell,2}(n) \geq d_{\ell,3}(n). 
	\]
\end{corollary}
Note that it is proved in \cite[Theorem 1.5]{SB2}  that for $\ell \geq 2$ and $t \geq 1$, $d_{\ell+1,t}(n) \geq d_{\ell,t}(n)$ for all $n \geq 0$.

The rest of the paper is organized as follows. In Section 2, we recall basic facts about the modular transformation for the partition generating function, Bernoulli polynomials, the digamma function, the Euler-Maclaurin summation formula, and an approximation of certain integrals using Bessel functions.  In Section 3, we establish the generating functions of $b_{\ell,t}(n)$ for $t=1,2,3$ and prove their asymptotics using the circle method.  In Section 4, we consider $\ell$-distinct partitions to derive the asymptotics of $d_{\ell,t}(n)$ for $t=1, 2, 3$.  Finally, we conclude with the proof of Theorem \ref{thm:bias_b_d}  in Section 5.

\section{Preliminaries}\label{S:Preliminaries}

Let $z\in\mathbb{C}$ with $\Re(z)>0$ and let $h,k\in\mathbb{N}_0$ with $0\leq h< k$ and $\gcd(h,k)=1$.
Then we have the transformation formula \cite[Sec. 5.2]{A} for $P(q):=\frac1{(q;q)_\infty}$ as
\begin{equation}\label{eq:P_trans}
	P(q) = \omega_{h,k}\sqrt ze^{\frac{\pi}{12k}\left(\frac1z-z\right)} P(q_1),
\end{equation}
where $q:=e^{\frac{2\pi i}k(h+iz)}$, $q_1:=e^{\frac{2\pi i}k(h'+\frac iz)}$ with $hh'\equiv-1\pmod k$, and $\omega_{h,k}:=e^{\pi i s(h,k)}$. Here, 
$s(h,k)$ is the {\it Dedekind sum} defined by
\[
	s(h,k):=\sum_{\mu\pmod{k}}\left(\!\!\left(\frac{\mu}{k}\right)\!\!\right)\left(\!\!\!\left(\frac{h\mu}{k}\right)\!\!\!\right)
\]
with
\[
	\left(\!\left(x\right)\!\right):=
	\begin{cases}
		x-\lfloor x\rfloor-\frac{1}{2}&\text{if }x\in\mathbb{R}\setminus\mathbb{Z},\\
		0&\text{if }x\in\mathbb{Z}.
	\end{cases}
\]

The {\it Bernoulli polynomials} $B_n(x)$ are defined as
\begin{equation}\label{eq:B_poly}
	\sum_{n \geq 0} \frac{B_n(x)}{n!} w^n  :=\frac{t e^{xw}}{e^w-1}.
\end{equation}
We recall the Euler-Maclaurin summation formula, which is modified from an exact formula given in \cite{Z}.
\begin{lemma}[{\cite[Theorem 1.3]{BJM}}]\label{lem:E_M_sum}
	Suppose that $0 \leq \theta < \frac{\pi}{2}$ and let  $D_\theta:=\{ r e^{i\alpha}: r \geq 0, |\alpha| \leq \theta\}$. Let $f:\mathbb{C} \to \mathbb{C}$ be holomorphic in a domain containing $D_\theta$ except for a simple pole at the origin, and assume that $f$ and all of its derivatives are of sufficient decay in $D_\theta$. If $f(w)=\sum_{n \geq -1} b_n w^n$ near $0$, then for $a \in \mathbb{R}\setminus \mathbb{Z}_{\leq 0}$ and $N \in \mathbb{N}_0$, uniformly, as $w \to 0$ in $D_\theta$,
	\begin{multline*}
		\sum_{m \geq 0} f\left( w(m+a)\right) = \frac{b_{-1} \Log \left( \frac1w\right)}{w} - \frac{b_{-1} \left( \gamma + \psi (a)\right)}{w} + \frac1w \int_0^\infty \left( f(x) - \frac{b_{-1}e^{-x}}{x}\right) dx\\
		 - \sum_{n=0}^{N-1} \frac{B_{n+1}(a)b_n}{n+1} w^n + \mathcal{O}\left(w^N\right).
	\end{multline*}
\end{lemma}

For $z \in \mathbb{C}$ with $\Re(z)>0$, the digamma function has the integral representation: 
\begin{equation}\label{eq:digamma_int}
	\psi(z) = \int_0^\infty \left( \frac{e^{-u}}{u} - \frac{e^{-zu}}{1-e^{-u}} \right) du.
\end{equation}
Also, the digamma function and its derivative have the series representation for $z \neq -1, -2, \dots$ :
\begin{equation}\label{eq:digamma_series}
	\psi(z) = -\gamma + \sum_{n \geq 0} \left( \frac1{n+1} - \frac1{n+z} \right), \qquad \psi'(z) = \sum_{n \geq 0} \frac1{(n+z)^2}.
\end{equation}

Next, we give an approximation of certain integrals using Bessel functions.  
\begin{lemma}[{\cite[Lemma 2.1]{BB}}]\label{lem:Bessel}
	Suppose that $k\in\mathbb{N}$, $s\in\mathbb{R}$, and let $\vartheta_1,\vartheta_2, A, B \in\mathbb{R}^+$ satisfy $k\ll\sqrt{n}$, $A\asymp\frac nk$, $B\ll\frac1k$, and $k\vartheta_1$, $k\vartheta_2\asymp\frac{1}{\sqrt{n}}$. Then we have
	\[
		\int_{\frac kn-ik\vartheta_1}^{\frac kn+ik\vartheta_2} z^{-s}e^{Az+\frac Bz}dz = 2\pi i\left(\frac AB \right)^\frac{s-1}{2}I_{s-1}\left(2\sqrt{AB}\right) + 
		\begin{cases}
			\mathcal{O}\left(n^{s-\frac12}\right) & \text{ if } s\geq 0,\\ 
			\mathcal{O}\left(n^{\frac{s-1}2}\right) & \text{ if } s <  0,\
		\end{cases}
	\]
	where $I_s(x)$ is the $I$-Bessel function.
\end{lemma}
The asymptotic of the $I$-Bessel function is 
\begin{equation}\label{Bessel_asymp}
	I_{s} (x) = \frac{e^{x}}{\sqrt{2 \pi x}} \left(1 + \frac{1-4s^2}{8x}+O\left(\frac1{x^2}\right)\right)
\end{equation}
as $x \to \infty$.

\section{$\ell$-regular partitions}

\subsection{Generating functions}

In this subsection, we establish the generating functions of $b_{\ell,t}(n)$ for $t=1, 2, 3$.

\begin{theorem}\label{thm:b_gen}
	Let $\ell \geq 2$ be an integer. For each $t=1$, $2$, and $3$, we have
	\begin{align*}
		\sum_{n \geq 0} b_{\ell,t}(n)q^n = \frac{(q^\ell;q^\ell)_\infty}{(q;q)_\infty} B_{\ell,t}(q)  
	\end{align*}
	where
	\begin{align*}
		B_{\ell,1}(q) &:= \frac{q}{1-q} - \frac{q^\ell}{1-q^\ell},\\
		B_{\ell,2}(q) &:=\frac{2q^2}{1-q^2} - \frac{q^\ell}{1-q^\ell} + \frac{q^{2\ell-1}-q^{2\ell}+q^{2\ell+1}}{1-q^{2\ell}},\\
		B_{\ell,3}(q) &:= \frac{3q^3}{1-q^3} - \frac{q^\ell}{1-q^\ell} + \frac{q^{2\ell-2}-q^{2\ell}+q^{2\ell+2}}{1-q^{2\ell}} \\
			&\quad  - \frac{q^{3\ell-3}-q^{3\ell-2}-q^{3\ell-1}+2q^{3\ell}-q^{3\ell+1}-q^{3\ell+2}+q^{3\ell+3}}{1-q^{3\ell}}.
	\end{align*}
\end{theorem}
\begin{proof}
First, we consider the generating function of $b_{\ell,2}(n)$. To count hooks of length $2$, there are two cases (see Figure \ref{fig:cases12}):
\begin{enumerate}
	\item partitions include a box $v$ with $a_\lambda(v)=1$ and $l_\lambda(v)=0$,
	\item partitions  include a box $v$ with $a_\lambda(v)=0$ and $l_\lambda(v)=1$.
\end{enumerate} 

\begin{figure}[ht!]
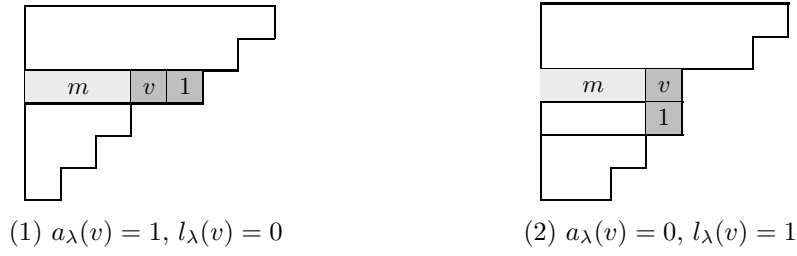

	\begin{subfigure}[b]{0.3\textwidth}
	\centering
	\scalebox{0.85}{
	\begin{tabular}{|lllllll}
		\hline
        		\hphantom{$v$}  & \hphantom{$v$}  &  \hphantom{$v$} &   &  \hphantom{$v$}  &   \hphantom{$v$}  & \multicolumn{1}{l|}{\hphantom{$v$}} \\ 
        		\cline{7-7} 
        		 &    &     &     &     & \multicolumn{1}{l|}{} &    \\ 
        		\cline{1-6}\hhline{--~}
		\multicolumn{3}{|c|}{\cellcolor{lightgray!30}$m$}   & \multicolumn{1}{l|}{\cellcolor{lightgray}$v$} & \multicolumn{1}{c|}{\cellcolor{lightgray}$1$} &     &           \\
		\cline{1-5}
		 &     & \multicolumn{1}{l|}{} &    &    &   &    \\ 
		\cline{3-3}
	  	 & \multicolumn{1}{l|}{} &     &     &    &     &   \\ 
		\cline{2-2}
		\multicolumn{1}{|l|}{} &     &   &    &     &   & \hphantom{$v$}   \\ 
		\cline{1-1}
	\end{tabular}}
	\caption*{(1) $a_\lambda(v)=1$, $l_\lambda(v)=0$}
	\end{subfigure}
	\hskip 50pt
	\begin{subfigure}[b]{0.3\textwidth}
	\centering
	\scalebox{0.85}{
	\begin{tabular}{|lllllll}
		\hline
        		\hphantom{$v$}  & \hphantom{$v$}  &  \hphantom{$v$}  &  \hphantom{$v$}  &  \hphantom{$v$}   & \hphantom{$v$}  & \multicolumn{1}{l|}{} \\ 
        		\cline{7-7} 
        		 &   &    &     &     & \multicolumn{1}{l|}{} &  \\ 
         	\cline{1-6}\hhline{--~}
		\multicolumn{3}{|c|}{\cellcolor{lightgray!30}$m$}  & \multicolumn{1}{l|}{\cellcolor{lightgray}$v$} & \multicolumn{1}{c}{} &  &   \\ 
		\cline{1-4}\hhline{--~}
		&  & \multicolumn{1}{l|}{} & \multicolumn{1}{l|}{\cellcolor{lightgray}$1$} &    &    &    \\ 
		\cline{3-4}\hhline{--~}
         	 & & \multicolumn{1}{l|}{} &   &   &   &  \\ 
          	\cline{3-3}
		& \multicolumn{1}{l|}{} &  &    &    &   &   \hphantom{$v$}   \\ 
		\cline{1-2}
	\end{tabular}}
	\caption*{(2) $a_\lambda(v)=0$, $l_\lambda(v)=1$}
	\end{subfigure}
	\caption{Cases for a hook of length $2$}
    	\label{fig:cases12}
\end{figure}

The partitions for the first case are $\ell$-regular partitions with at least one part of size $m+2$ and no part of size $m+1$ for all $m \geq 0$ where $m+2 \not \equiv 0 \pmod \ell$. Thus, the generating function is
\begin{multline*}
	\sum_{\substack{m \geq 0 \\ m \not\equiv -1, -2 \!\!\!\!\pmod \ell\\}} q^{m+2} \frac{(q^\ell;q^\ell)_\infty}{(q;q)_\infty} (1-q^{m+1})
	+ \sum_{\substack{m \geq 0 \\ m \equiv -1 \!\!\!\!\pmod \ell}} q^{m+2} \frac{(q^\ell;q^\ell)_\infty}{(q;q)_\infty}\\
	= \frac{(q^\ell;q^\ell)_\infty}{(q;q)_\infty} \left( \frac{q^2}{1-q^2} - \frac{q^\ell}{1-q^\ell} + \frac{q^{2\ell-1}+q^{2\ell+1}}{1-q^{2\ell}} \right).
\end{multline*}
For the second case, we have $\ell$-regular partitions with at least two parts of size $m+1$, whose generating function is
\[
	\sum_{\substack{m \geq 0 \\ m \not\equiv -1 \!\!\!\!\pmod \ell}} q^{2(m+1)} \frac{(q^\ell;q^\ell)_\infty}{(q;q)_\infty} = \frac{(q^\ell;q^\ell)_\infty}{(q;q)_\infty} \left( \frac{q^2}{1-q^2} - \frac{q^{2\ell}}{1-q^{2\ell}} \right).
\]
Adding the generating functions for two cases yields the generating function of $b_{\ell,2}(n)$.

Similarly, to count hooks of length $3$, we consider three cases as in Figure \ref{fig:3cases}. Note that for the second case, there might be a box below the arm of the box $v$, which is indicated by $*$.
\begin{figure}[ht!]
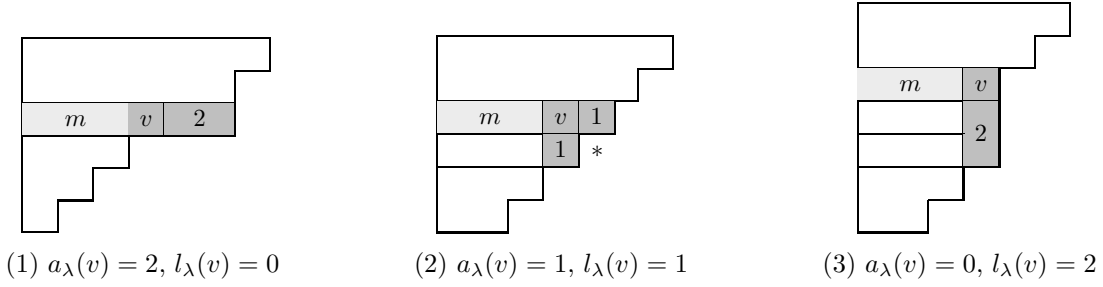

	\begin{subfigure}[b]{0.3\textwidth}
	\centering
	\scalebox{0.85}{
	\begin{tabular}{|lllllll}
		\hline
		\hphantom{$v$} &  \hphantom{$v$}  &\hphantom{$v$}  &\hphantom{$v$}   &\hphantom{$v$} &  \hphantom{$v$}  & \multicolumn{1}{l|}{} \\ \cline{7-7} 
		&   &    &   &   & \multicolumn{1}{l|}{}   &    \\ \hhline{------~}
\multicolumn{3}{|c|}{\cellcolor{lightgray!30}$m$} & \multicolumn{1}{l|}{\cellcolor{lightgray}$v$} & \multicolumn{2}{c|}{\cellcolor{lightgray}$2$} &  \\ \cline{1-6}
		&    & \multicolumn{1}{l|}{} &    &    &    &   \\ \cline{3-3}
		& \multicolumn{1}{l|}{} &  &   &   &    &    \\ \cline{2-2}
		\multicolumn{1}{|l|}{} &   &   &    &   &    &    \hphantom{$v$}   \\ \cline{1-1}
	\end{tabular}}
	\caption*{(1) $a_\lambda(v)=2$, $l_\lambda(v)=0$}
	\end{subfigure}
	\hskip 10pt
	\begin{subfigure}[b]{0.3\textwidth}
	\centering
	\scalebox{0.85}{
	\begin{tabular}{|lllllll}
		\hline
 		\hphantom{$v$} &  \hphantom{$v$}   &\hphantom{$v$}  &\hphantom{$v$}   &\hphantom{$v$} &  & \multicolumn{1}{l|}{} \\ \cline{7-7} 
 		&   &   &   &    & \multicolumn{1}{l|}{}  &  \\ \hhline{------~}
		\multicolumn{3}{|c|}{\cellcolor{lightgray!30}$m$}  & \multicolumn{1}{l|}{\cellcolor{lightgray}$v$}  & \multicolumn{1}{c|}{\cellcolor{lightgray}$1$} &  &  \\ \hhline{-----~} 
 		&   & \multicolumn{1}{l|}{} & \multicolumn{1}{l|}{\cellcolor{lightgray}$1$} & \multicolumn{1}{c}{$*$}&   &    \\  \hhline{----~} 
 		&   & \multicolumn{1}{l|}{} &    &    &   &   \\ \cline{3-3}
		& \multicolumn{1}{l|}{} &  &   &   &   &  \hphantom{$v$}  \\ \cline{1-2}
	\end{tabular}}
	\caption*{(2) $a_\lambda(v)=1$, $l_\lambda(v)=1$}
	\end{subfigure}
	\hskip 10pt
	\begin{subfigure}[b]{0.3\textwidth}
	\centering
	\scalebox{0.85}{
	\begin{tabular}{|llllll}
		\hline
		\hphantom{$v$} &  \hphantom{$v$}  &\hphantom{$v$}  &\hphantom{$v$}   &\hphantom{$v$}  & \multicolumn{1}{l|}{} \\ \cline{6-6} 
		 &   &   &    & \multicolumn{1}{l|}{} &   \\ \hhline{-----~}
		\multicolumn{3}{|c|}{\cellcolor{lightgray!30}$m$}  & \multicolumn{1}{l|}{\cellcolor{lightgray}$v$}  & \multicolumn{1}{c}{}  &   \\ \hhline{----~}
		&   & \multicolumn{1}{l|}{} & \multicolumn{1}{l|}{\cellcolor{lightgray}} &    &     \\ \hhline{---}
		&   & \multicolumn{1}{l|}{} & \multicolumn{1}{l|}{\cellcolor{lightgray}\multirow{-2}{*}{$2$}}  &    &   \\ \hhline{----~}
 		&   & \multicolumn{1}{l|}{} &   &   &    \\ \cline{3-3}
 		& \multicolumn{1}{l|}{} &   &    &    &   \hphantom{$v$}   \\ \cline{1-2}
	\end{tabular}}
	\caption*{(3) $a_\lambda(v)=0$, $l_\lambda(v)=2$}
	\end{subfigure}
	\caption{Cases for a hook of length $3$}
    	\label{fig:3cases}
\end{figure}
Then the generating functions for three cases are 
\begin{align*}
	G_1(q)=&\sum_{\substack{m \geq 0 \\ m \not\equiv -1,-2,-3 \!\!\!\!\pmod \ell}} q^{m+3} \frac{(q^\ell;q^\ell)_\infty}{(q;q)_\infty}  (1-q^{m+1})(1-q^{m+2})\\
	&\quad+ \sum_{\substack{m \geq 0 \\ m \equiv -1 \!\!\!\!\pmod \ell}} q^{m+3} \frac{(q^\ell;q^\ell)_\infty}{(q;q)_\infty} (1-q^{m+2})
	+ \sum_{\substack{m \geq 0 \\ m \equiv -2 \!\!\!\!\pmod \ell}} q^{m+3} \frac{(q^\ell;q^\ell)_\infty}{(q;q)_\infty} (1-q^{m+1}),\\
	G_2(q)=&\sum_{\substack{m \geq 0 \\ m \not\equiv -1, -2 \!\!\!\!\pmod \ell\\}} (1+q)q^{2m+3} \frac{(q^\ell;q^\ell)_\infty}{(q;q)_\infty} (1-q^{m+1})
	+ \sum_{\substack{m \geq 0 \\ m \equiv -1 \!\!\!\!\pmod \ell}} q^{2(m+2)} \frac{(q^\ell;q^\ell)_\infty}{(q;q)_\infty},\\
	G_3(q)=&\sum_{\substack{m \geq 0 \\ m \not\equiv -1 \!\!\!\!\pmod \ell}} q^{3(m+1)} \frac{(q^\ell;q^\ell)_\infty}{(q;q)_\infty},
\end{align*}
which provide the generating function of $b_{\ell,3}(n)$.
The generating function of $b_{\ell,1}(n)$ can be found in the same way.
\end{proof}

\begin{remark}
	Note that the generating function of $b_{\ell,1}(n)$ is also given in \cite[Theorem 1.2]{SB1}.
	We also note that by Theorem \ref{thm:b_gen}, the generating function of $b_{4,2}(n)$ is
	\begin{align*}
		\sum_{n \geq 0} b_{4,2}(n)q^n &= \frac{(q^4;q^4)_\infty}{(q;q)_\infty} \left( \frac{2q^2}{1-q^2} - \frac{q^4}{1-q^4} + \frac{q^{7}-q^{8}+q^{9}}{1-q^{8}} \right),
	\end{align*}
	which is stated in \cite[Theorem 1.3]{SB1} incorrectly.
\end{remark}

\subsection{Asymptotics of generating functions}

In this subsection, we evaluate the asymptotics of generating functions. We let $q=e^{\frac{2\pi i}k(h+iz)}$ where $z\in\mathbb{C}$ with $\Re(z)>0$ and $h,k\in\mathbb{N}_0$ with $0\leq h< k$ and $\gcd(h,k)=1$.  Let $\zeta_k := e^{\frac{2\pi i}{k}}$ for $k \in \mathbb{N}$.

\begin{lemma}\label{lem:xi}
	Let $\ell \geq 2$ be an integer.
	\begin{enumerate}
		\item If $\ell \nmid k$, we have that as $z \to 0$,
		\[
			\frac{P(q)}{P(q^\ell)} = \frac{\omega_{h,k}}{\sqrt{\ell} \, \omega_{\ell h,k}} e^{\frac{(\ell-1)\pi}{12k} \left( \frac{1}{\ell z} + z \right)} \left( 1 + \mathcal{O}\left( e^{-\frac{2\pi}{kz}}\right)\right).
		\]
		\item If $\ell |k$, we have that as $z \to 0$,
		\[
			\frac{P(q)}{P(q^\ell)} \ll e^{- \frac{(\ell-1)\pi}{12k} \left( \frac1z -z \right)}.
		\]
	\end{enumerate}
\end{lemma}
\begin{proof}
(1) Assume $\ell \nmid k$. Then $q^\ell=e^{\frac{2\pi i}{k} \left( \ell h + i \ell z\right)}$. Applying \eqref{eq:P_trans} yields that
\begin{align*}
	\frac{P(q)}{P(q^\ell)} &= \frac{\omega_{h,k}\sqrt{z} e^{\frac{\pi}{12k}\left(\frac1z-z\right)} P(q_1)}{\omega_{\ell h,k}\sqrt{\ell z}e^{\frac{\pi}{12k}\left(\frac1{\ell z}-\ell z\right)} P\left(q_1^{\frac1\ell}\right)}\\
	 &= \frac{\omega_{h,k}}{\sqrt{\ell}\, \omega_{\ell h,k}} e^{\frac{(\ell-1)\pi}{12k} \left( \frac{1}{\ell z} + z \right)}  \frac{P(q_1)}{P\left(q_1^{\frac1\ell}\right)} = \frac{\omega_{h,k}}{\sqrt{\ell}\, \omega_{\ell h,k}} e^{\frac{(\ell-1)\pi}{12k} \left( \frac{1}{\ell z} + z \right)} \left( 1 + \mathcal{O}\left( q_1\right)\right).
\end{align*}

(2) If $\ell | k$, then $q^\ell=e^{\frac{2\pi i}{\frac{k}{\ell}} \left( h + iz\right)}$. Thus, by \eqref{eq:P_trans}
\[
	\frac{P(q)}{P(q^\ell)} = \frac{\omega_{h,k} \sqrt{z} e^{\frac{\pi}{12k}\left(\frac1z-z\right)} P(q_1)}{\omega_{h,k} \sqrt{z} e^{\frac{\pi}{12\frac{k}{\ell}}\left(\frac1z-z\right)} P(q_1^\ell)} = e^{- \frac{(\ell-1)\pi}{12k} \left( \frac1z -z \right)} \frac{P(q_1)}{P\left(q_1^\ell \right)} \ll e^{- \frac{(\ell-1)\pi}{12k} \left( \frac1z -z \right)}. \qedhere
\]
\end{proof}

Let $\delta_S:=1$ if the statement $S$ holds and $\delta_S:=0$ otherwise.

\begin{lemma}\label{lem:B2}
	Let $\ell \geq 2$ be an integer. As $z \to 0$,
	\[
		B_{\ell,2}(q) =\left( \delta_{k|2} - \frac{\delta_{k|\ell}}{\ell} + \frac{\left(2\cos\left( \frac{2\pi h}{k}\right)-1 \right) \delta_{k|2\ell}}{2\ell} \right) \frac{k}{2 \pi z}  + \mathcal{O}(k).
	\]
	In particular, when $k=1$, as $z \to 0$,
	\[
		B_{\ell,2}(q) =\left( 1- \frac1{2\ell} \right) \frac{1}{2 \pi z} -1 + \mathcal{O}(z).
	\]
\end{lemma}
\begin{proof}
Let $m \geq 1$ be an integer.
If $k|m$, we use \eqref{eq:B_poly} to obtain that as $z \to 0$,
\begin{align}
	\frac{q^a}{1-q^m} = \frac{\zeta_{k}^{ah} e^{\frac{2\pi (m-a)z}{k}} }{e^{\frac{2\pi m z}{k}}-1} 
	&= \zeta_{k}^{ah} \sum_{n \geq 0} \frac{B_n\left(1-\frac{a}{m}\right)}{n!} \left( \frac{2\pi mz}{k} \right)^{n-1} \nonumber\\
	&= \zeta_k^{ah} \left(\frac{k}{2\pi m z}+ \frac{1}{2} - \frac{a}{m} \right) + \mathcal{O}\left( \frac{z}{k} \right). \label{eq:Bernoulli1} 
\end{align}
If $k \nmid m$, then as $z \to 0$,
\begin{equation}\label{eq:Bernoulli2}
	\frac{q^m}{1-q^m} = \frac{1}{\zeta_{k}^{-mh} e^{\frac{2\pi mz}{k}}-1} \ll \frac{1}{\zeta_{k}^{-mh} -1}  \ll k,
\end{equation}
where we use the following for the last inequality:
\begin{equation}\label{eq:zeta_bound}
	\left| \zeta_k^{-mh}- 1 \right|^2 = 2-2 \cos\left( \frac{2\pi mh}{k}\right) \geq 2- 2\cos\left( \frac{2\pi}{k}\right) \gg \frac1{k^2}
\end{equation}
since $k \nmid m$, $\gcd(h,k)=1$, and $0<\frac{2\pi}{k} \leq \pi$ for $k \geq 2$.
Therefore, the desired results follow from \eqref{eq:Bernoulli1} and \eqref{eq:Bernoulli2}.
\end{proof}

Similarly, we can find the asymptotic formula for $B_{\ell,1}(q)$ and $B_{\ell,3}(q)$.

\begin{lemma}
	Let $\ell \geq 2$ be an integer. As $z \to 0$, 
	\begin{align*}
		B_{\ell,1}(q) =& \left( \delta_{k=1} - \frac{\delta_{k|\ell}}{\ell}  \right) \frac{k}{2 \pi z} + \mathcal{O}(k),\\
		B_{\ell,3}(q) =&\left( \delta_{k|3} - \frac{\delta_{k|\ell}}{\ell} + \frac{\left( 2\cos\left(\frac{4\pi h}{k}\right) -1 \right) \delta_{k|2\ell}}{2\ell} - \frac{2 \left(\cos\left(\frac{6\pi h}{k}\right)-\cos\left(\frac{4\pi h}{k}\right)-\cos\left(\frac{2\pi h}{k}\right)+1 \right) \delta_{k|3\ell}}{3\ell}\right)\\
		& \times \frac{k}{2 \pi z}  + \mathcal{O}(k).
	\end{align*}
	In particular, when $k=1$, as $z \to 0$,
	\[
		B_{\ell,1}(q) =\left( 1- \frac1{\ell} \right) \frac{1}{2 \pi z} + \mathcal{O}(z) \qquad \text{and} \qquad B_{\ell,3}(q) =\left( 1- \frac1{2\ell} \right) \frac{1}{2 \pi z} -\frac32 + \mathcal{O}(z).
	\]
\end{lemma}

\subsection{Proof of Theorem \ref{thm:b_asymp}}

We now apply the circle method to derive the asymptotics of $b_{\ell,t}(n)$ for $t=1, 2, 3$.

\begin{proof}[Proof of Theorem \ref{thm:b_asymp}]
We will give a proof for the asymptotic of $b_{\ell,2}(n)$ as $n \to \infty$. The asymptotics of $b_{\ell,1}(n)$ and $b_{\ell,3}(n)$ can be obtained in a similar way.

Let $0\le h<k\le N$ with $\gcd(h,k)=1$, and $z=\frac kn -ik\phi$ with $-\vartheta_{h,k}'\leq \phi \leq \vartheta_{h,k}''$, where
\[
	\vartheta_{0,1}' :=\frac1{N+1}, \qquad \vartheta_{h,k}' := \frac{1}{k(k_1+k)} ~~\text{ for } h>0,\qquad\text{and}\qquad \vartheta_{h,k}'' := \frac{1}{k(k_2+k)}.
\]
Here, $\frac{h_1}{k_1}<\frac hk<\frac{h_2}{k_2}$ are adjacent Farey fractions in the Farey sequence of order $N:=\lfloor \sqrt n \rfloor$. From the theory of Farey fractions, it is well-known that
\begin{equation}\label{Fbound}
	\frac1{k+k_j} \le \frac1{N+1} \quad \text{for } j\in\{1, 2\}.
\end{equation}
Moreover, we have
\begin{equation}\label{zbound}
	\Re(z)=\frac{k}{n},\quad \Re \left(\frac{1}{z}\right) \geq \frac{k}{2}, \quad |z| \ll \frac1{\sqrt n}, \quad \text{and} \quad |z| \geq \frac kn.
\end{equation}

We employ Cauchy's integral formula to obtain for $q=e^{\frac{2\pi i}k(h+iz)}$ that
\begin{align}
	b_{\ell,2}(n) &=\frac{1}{2\pi i} \int_{|q|=e^{-\frac{2\pi}{n}}} \frac{P(q)}{P(q^\ell)}B_{\ell,2}(q) q^{-n-1} dq \nonumber\\
	&= \sum_{\substack{0\le h<k\le N\\\gcd(h,k)=1}} e^{-\frac{2\pi i n h}k} \int_{-\vartheta_{h,k}'}^{\vartheta_{h,k}''} \frac{P(q)}{P(q^\ell)}B_{\ell,2}(q) e^\frac{2\pi n z}k d\phi. \label{eq:b2}
\end{align}
We now split \eqref{eq:b2} into the term when $k=1$, the sum over $\ell |k$ and $k \geq 2$, and the sum over $\ell \nmid k$ as
\[
	b_{\ell,2}(n) = M + E_1 + E_2,
\]
where
\begin{align*}	
	M &:= \int_{-\vartheta_{0,1}'}^{\vartheta_{0,1}''} \frac{P(q)}{P(q^\ell)}B_{\ell,2}(q) e^{2\pi n z} d\phi,\\
	E_1 &:= \sum_{\substack{0 < h<k\le N\\\gcd(h,k)=1\\ \ell \nmid k}} e^{-\frac{2\pi i n h}k} \int_{-\vartheta_{h,k}'}^{\vartheta_{h,k}''} \frac{P(q)}{P(q^\ell)}B_{\ell,2}(q) e^\frac{2\pi n z}k d\phi,\\
	E_2 &:= \sum_{\substack{0\le h<k\le N\\\gcd(h,k)=1\\ \ell | k }} e^{-\frac{2\pi i n h}k} \int_{-\vartheta_{h,k}'}^{\vartheta_{h,k}''} \frac{P(q)}{P(q^\ell)}B_{\ell,2}(q) e^\frac{2\pi n z}k d\phi.
\end{align*}

By Lemmas \ref{lem:xi} (1) and \ref{lem:B2}, we find that
\begin{multline}\label{eq:sgima1}
	M = \frac1{2 \pi \sqrt{\ell}} \left( 1 -\frac1{2\ell} \right) \mathcal{I}_{1}\left( 2\pi \left(n+\frac{\ell-1}{24}\right),\frac{(\ell-1)\pi}{12 \ell}\right) \\
	- \frac1{\sqrt{\ell}}  \mathcal{I}_{0}\left( 2\pi \left(n+\frac{\ell-1}{24}\right),\frac{(\ell-1)\pi}{12 \ell }\right) 
	+ \mathcal{O}\left(\mathcal{I}_{-1}\left(2\pi \left(n+\frac{\ell-1}{24}\right), \frac{(\ell-1)\pi}{12 \ell }\right) \right),
\end{multline}
where, for $s \in \mathbb{R}$,
\[
	\mathcal{I}_{s}(A,B) := \int_{-\vartheta_{h,k}'}^{\vartheta_{h,k}''} z^{-s} e^{Az+\frac{B}{z}} d\phi
	= \frac1{ik} \int_{\frac kn-\frac{ik}{k\left(k+k_2\right)}}^{\frac kn+\frac{ik}{k\left(k+k_1\right)}} z^{-s} e^{Az+\frac Bz} dz.
\]
Similarly, by Lemmas \ref{lem:xi} (1) and \ref{lem:B2}, 
\begin{equation}\label{eq:sigma2}
	E_1 \ll \sum_{\substack{0 < h<k\le N\\\gcd(h,k)=1\\ \ell \nmid k}} k  \mathcal{I}_{1}\left( \frac{2\pi}{k} \left(n+\frac{\ell-1}{24}\right),\frac{(\ell-1)\pi}{12 \ell k}\right),
\end{equation}

Next, we apply Lemma \ref{lem:Bessel} with $\vartheta_1=\frac1{k(k+k_2)}$, $\vartheta_2=\frac1{k(k+k_1)}$, $A=\frac{2\pi}{k} (n+\frac{\ell-1}{24})$, and $B= \frac{(\ell-1)\pi}{12 \ell k}$ to get
\begin{equation}\label{eq:Bessel}
	\mathcal{I}_ s(A, B) = \frac{2\pi}{k} \left(\frac{24\ell}{\ell-1}  \left( n+ \frac{\ell-1}{24} \right) \right)^{\frac{s-1}{2}} \! I_{s-1}\left(\frac{\pi}{k} \sqrt{\frac{2(\ell-1)}{3\ell}\left( n+\frac{\ell-1}{24}\right)} \right) +  		\begin{cases}
			\mathcal{O}\left(\frac{n^{s-\frac12}}{k}\right) &\!\!\! \text{if } s\geq 0,\\ 
			\mathcal{O}\left(\frac{n^{\frac{s-1}2}}{k}\right) &\!\!\! \text{if } s <  0.
		\end{cases}
\end{equation}
Employing  \eqref{eq:Bessel} and the asymptotic of $I$-Bessel function \eqref{Bessel_asymp} to \eqref{eq:sgima1}, we obtain that
\begin{align*}
	M &= \frac1{\pi} \left( 1- \frac1{2\ell} \right)  \left( \frac{3}{8\ell(\ell-1)} \right)^{\frac14} \left( n+ \frac{\ell-1}{24}\right)^{-\frac14} e^{\pi \sqrt{\frac{2(\ell-1)}{3\ell}\left( n+\frac{\ell-1}{24}\right)}}\\
	&\hskip 200pt \times \left( 1+ \frac{1}{8\pi} \left( \frac{3\ell}{2(\ell-1)}\right)^{\frac12}  \left( n + \frac{\ell-1}{24}\right)^{-\frac12} + O\left(\frac{1}{n}\right) \right)\\
	&\quad -  \frac1{2\ell} \left( \frac{\ell(\ell-1)}{6} \right)^{\frac14} \left( n+ \frac{\ell-1}{24}\right)^{-\frac34} e^{\pi \sqrt{\frac{2(\ell-1)}{3\ell}\left( n+\frac{\ell-1}{24}\right)}}\left( 1+  O\left(\frac{1}{\sqrt n}\right) \right)\\
	&=  \frac1{\pi} \left( 1- \frac1{2\ell} \right)  \left( \frac{3}{8\ell(\ell-1)} \right)^{\frac14} n^{-\frac14} e^{\pi \sqrt{\frac{2n}{3} \left(1- \frac1\ell \right)}}   \left( 1+ \frac{\pi(\ell-1)}{24}  \left(\frac{\ell-1}{6\ell}\right)^{\frac12} n^{-\frac12}+ \mathcal{O}\left( \frac1n \right)\right)\\
	&\hskip 28pt \times \left[ \left( 1+\frac{1}{8\pi} \left( \frac{3\ell}{2(\ell-1)}\right)^{\frac12}  n^{-\frac12} + O\left(\frac{1}{n}\right) \right) - \frac{\pi}{2\ell-1} \left( \frac{2\ell(\ell-1)}{3} \right)^{\frac12} n^{-\frac12} \left( 1+  O\left(\frac{1}{\sqrt n}\right) \right) \right]\\
	&= \frac1{\pi} \left( 1- \frac1{2\ell} \right)  \left( \frac{3}{8\ell(\ell-1)} \right)^{\frac14} n^{-\frac14} e^{\pi \sqrt{\frac{2n}{3} \left(1- \frac1\ell \right)}}  \\
	&\hskip 80pt  \times  \left[ 1 +\left( \frac{\pi(\ell-1)}{24} \sqrt{\frac{\ell-1}{6\ell}} - \frac{2\pi}{2\ell-1} \sqrt{\frac{\ell(\ell-1)}{6}} + \frac1{16\pi} \sqrt{\frac{6\ell}{\ell-1}} \right) \frac1{\sqrt{n}} + \mathcal{O}\left( \frac1{n} \right) \right],
\end{align*}
which follows by expanding, for $r \in \mathbb{R}^+$,
\begin{align*}
	 \left( n+ \frac{\ell-1}{24}\right)^{-r} &= n^{-r} \left( 1 + \mathcal{O} \left( \frac1n \right) \right),\\
	e^{\pi \sqrt{\frac{2(\ell-1)}{3\ell}\left( n+\frac{\ell-1}{24}\right)}} &= e^{\pi \sqrt{\frac{2n}{3} \left(1- \frac1\ell \right)}} \left( 1+ \frac{\pi (\ell-1)}{24}  \sqrt{ \frac{\ell-1}{6 \ell n}}+ \mathcal{O}\left( \frac1n \right)\right).
\end{align*}
We also get the bound of $E_1$ by using \eqref{eq:Bessel} and \eqref{Bessel_asymp} to \eqref{eq:sigma2}, 
\[
	E_1 \ll n^{-\frac14} e^{\frac{\pi}{2} \sqrt{\frac{2(\ell-1)}{3\ell}\left( n+\frac{\ell-1}{24}\right)}} \sum_{0< h<k\le N}  \sqrt{k} \ll n e^{\frac{\pi}{2} \sqrt{\frac{2n}{3}\left( 1-\frac1\ell \right)}}.
\]

Lastly, by Lemmas \ref{lem:xi} (2) and \ref{lem:B2} with \eqref{Fbound} and \eqref{zbound}, we bound $E_2$ as
\begin{align*}
	E_2  \ll & \sum_{\substack{0\le h<k\le N\\\gcd(h,k)=1}} \left(\vartheta_{h,k}' +\vartheta_{h,k}''\right) \max_{z}\left|\frac{k}{z} e^{-\frac{(\ell-1)\pi}{12kz}}\right|  \ll \frac1{N+1}  \sum_{0\le h<k\le N} \frac{n}{k} e^{-\frac{(\ell-1)\pi}{24}} \ll n.
\end{align*}
Combining the estimations of $M$, $E_1$, and $E_2$ provides the asymptotic formula of $b_{\ell,2}(n)$ as $n \to \infty$. 
\end{proof}

\section{$\ell$-distinct partitions}

\subsection{Generating functions}

In this subsection, we find the generating functions of $d_{\ell,t}(n)$ for $t=1, 2, 3$. 

\begin{theorem}
	Let $\ell \geq 2$ be an integer. For $t=1$, $2$, and $3$, we have
	\[
		\sum_{n \geq 0} d_{\ell,t}(n) q^n = \frac{(q^\ell;q^\ell)_\infty}{(q;q)_\infty} D_{\ell,t}(q),
	\]
	where
	\begin{align*}
		D_{\ell,1}(q) &= \sum_{m \geq 0} q^{m+1} \frac{1-q^{(\ell-1)(m+1)}}{1-q^{\ell(m+1)}},\\
		D_{\ell,2}(q) &= \sum_{m \geq 0} \left(  q^{m+2}\frac{(1-q^{m+1})(1-q^{(\ell-1)(m+2)})}{(1-q^{\ell(m+1)})(1-q^{\ell(m+2)})} + q^{2m+2} \frac{1-q^{(\ell-2)(m+1)}}{1-q^{\ell(m+1)}} \right),\\
		D_{\ell,3}(q) &= \sum_{m \geq 0}\left( q^{m+3}\frac{(1-q^{m+1})(1-q^{m+2})(1-q^{(\ell-1)(m+3)})}{(1-q^{\ell(m+1)})(1-q^{\ell(m+2)})(1-q^{\ell(m+3)})}  \right.\\
				&\hskip 34pt \left.+ q^{2m+3} \frac{1-q^{m+1}}{1-q^{\ell(m+1)}} \frac{(1-q^{(\ell-1)(m+2)}) + q (1-q^{(\ell-2)(m+1)})}{1-q^{\ell(m+2)}}  + q^{3m+3} \frac{1-q^{(\ell-3)(m+1)}}{1-q^{\ell(m+1)}} \right).
	\end{align*}
\end{theorem}
\begin{proof}
We give a proof for the generating function of $d_{\ell,2}(n)$. Similarly, we can establish the generating functions of $d_{\ell,1}(n)$ and $d_{\ell,3}(n)$. 

As in the proof of Theorem \ref{thm:b_gen}, we consider two cases with $\ell$-distinct partition. For the case when $a_\lambda(v)=1$ and $l_\lambda(v)=0$, we have $\ell$-distinct partitions with at least one part of size $m+2$ and no part of size $m+1$ for all $m \geq 0$, whose generating function is
\[
	\sum_{m \geq 0} q^{m+2} \frac{(q^\ell;q^\ell)_\infty}{(q;q)_\infty} \frac{1-q^{m+1}}{1-q^{\ell(m+1)}} \frac{1-q^{(\ell-1)(m+2)}}{1-q^{\ell(m+2)}}.
\]
For another case when $a_\lambda(v)=0$ and $l_\lambda(v)=1$, the partitions are $\ell$-distinct partitions with at least two parts of size $m+1$, which is generated by
\[
	\sum_{m \geq 0} q^{2(m+1)} \frac{(q^\ell;q^\ell)_\infty}{(q;q)_\infty} \frac{1-q^{(\ell-2)(m+1)}}{1-q^{\ell(m+1)}}.
\]
Combining the generating functions for two cases provides the generating function of $d_{\ell,2}(n)$.
\end{proof}

\subsection{Asymptotics for the generating functions}

We now prove the asymptotics of $D_{\ell,t}(q)$ for $t=1, 2, 3$. 
Let $q=e^{\frac{2\pi i}k(h+iz)}$ where $z\in\mathbb{C}$ with $\Re(z)>0$ and $h,k\in\mathbb{N}_0$ with $0\leq h< k$ and $\gcd(h,k)=1$.

\begin{lemma}\label{lem:D2} 
	Let $\ell \geq 2$ be a fixed integer and $0 \leq \theta < \frac{\pi}{2}$. Then when $k=1$, as $z \to 0$  in $D_\theta$, 
	\[
		D_{\ell,2}(q) = \left[ 1- \gamma - \left(1-\frac2\ell \right) \psi\left( \frac1\ell \right) -\frac2\ell \psi\left( \frac2\ell \right) \right] \frac{1}{2\pi \ell z} + \mathcal{O}\left(1\right).
	\]
	And we have the bounds of $D_{\ell,2}(q)$ for $k\geq2$ as follows:
	\[
		D_{\ell,2}(q) \ll \frac{k^3}{z^3}  \quad \text{for } k\geq 2 \text{ and } k|\ell, \qquad  D_{\ell,2}(q) \ll \frac{k^4}{z}  \quad \text{for } k\nmid \ell.
	\]	
\end{lemma}
\begin{proof}
Using the partial fraction 
\[
	\frac1{(1-q^{\ell(m+1)})(1-q^{\ell(m+2)})} = \frac1{1-q^\ell} \left( \frac1{1-q^{\ell(m+1)}} - \frac{q^\ell}{1-q^{\ell(m+2)}} \right),
\] 
we rewrite $D_{\ell,2}(q)$ as
\begin{align*}
	D_{\ell,2}(q) =& \frac{q-q^\ell}{1-q^\ell} \sum_{m \geq 0} \frac{q^{m+1}}{1-q^{\ell(m+1)}} + \left(1-\frac{q-q^{\ell-1}}{1-q^{\ell}} \right) \sum_{m \geq 0} \frac{q^{2(m+1)}}{1-q^{\ell(m+1)}} - \sum_{m \geq 0} \frac{q^{\ell(m+1)}}{1-q^{\ell(m+1)}}  \\
	&\hskip 275pt - \frac{q^{\ell-1}-q^{\ell}}{1-q^\ell} \sum_{m \geq 0} \frac{q^{(\ell+1)(m+1)}}{1-q^{\ell(m+1)}} \\
	=& \big(f_{h,k,1}(z) - f_{h,k,\ell}(z)\big) F_{h,k,1}(z) + \big(1- f_{h,k,1}(z) + f_{h,k,\ell-1}(z)\big) F_{h,k,2}(z) - F_{h,k,\ell}(z) \\
	&\hskip 240pt - \big(f_{h,k,\ell-1}(z) - f_{h,k,\ell}(z)\big)  F_{h,k,\ell+1} (z),
\end{align*}
where
\[
	f_{h,k,a}(z) := \frac{\left(\zeta_{k}^h e^{-\frac{2\pi z}{k}} \right)^a}{1-\left(\zeta_{k}^h e^{-\frac{2\pi z}{k}} \right)^\ell} \quad \text{ and }\quad 
	F_{h,k,a}(z):=\sum_{m \geq 0} f_{h,k,a} \left( (m+1) z\right).  
\]

First, assume that $k|\ell$. From \eqref{eq:Bernoulli1}, we have that as $z \to 0$,
\begin{equation}\label{eq:f_asymp}
	f_{h,k,a}(z) =  \zeta_k^{ah} \left(\frac{k}{2\pi \ell z}+ \frac{1}{2} - \frac{a}{\ell} \right) + \mathcal{O}\left( \frac{z}{k} \right).
\end{equation}
Applying Lemma \ref{lem:E_M_sum} with $\psi(1)=-\gamma$ provides that as $z \to 0$ in $D_\theta$,
\[
	F_{h,k,a}(z)=\sum_{m \geq 0} f_{h,k,a} \left( (m+1) z\right)  =  - \frac{k \zeta_k^{ah}   \Log z }{2\pi \ell z} +\frac{I^*_{h,k,a}}{z} - \frac{\zeta_k^{ah}}2 \left( \frac12 -\frac{a}{\ell}\right) +\mathcal{O}\left( \frac{z}{k} \right),
\]
where 
\[
	I^*_{h,k,a} = \int_{0}^\infty \left( f_{h,k,a}(u) - \frac{k \zeta_k^{ah} e^{-u}}{2\pi \ell u} \right) du.
\]
We evaluate $I^*_{h,k,a}$ for $a>0$, with using \eqref{eq:digamma_int},
\begin{align*}
	I^*_{h,k,a} &= \zeta_k^{ah} \int_0^\infty \left( \frac{e^{-\frac{2\pi a u}{k}}}{1-e^{-\frac{2\pi \ell u}{k}}} - \frac{k e^{-u}}{2\pi \ell u}  \right) du\\
	&= \frac{k \zeta_k^{ah}}{2\pi \ell} \left( \int_0^\infty \left( \frac{e^{-\frac{a}{\ell}u}}{1-e^{-u}} - \frac{e^{-u}}{u} \right) du + \int_0^\infty  \left( \frac{e^{-u}}{u} - \frac{e^{-\frac{ku}{2\pi \ell}}}{u}  \right) du \right)\\
	& = -\frac{k \zeta_k^{ah}}{2\pi \ell} \left(\psi\left(\frac{a}\ell \right) +  \log \left(\frac{2\pi \ell}{k}\right)  \right), 
\end{align*}
which gives that
\begin{equation}\label{eq:F}
	F_{h,k,a}(z)=  -\zeta_k^{ah}  \left[ \frac{k}{2\pi \ell z} \left( \Log \left( \frac{2\pi \ell z}{k}\right) + \psi \left( \frac{a}{\ell} \right) \right)  + \frac12  \left( \frac12 -\frac{a}{\ell}\right) \right] +\mathcal{O}\left( \frac{z}{k} \right).
\end{equation}
Therefore, from \eqref{eq:f_asymp} and \eqref{eq:F}, we arrive at the asymptotic of $D_{\ell,2}(q)$ for $k=1$
\begin{align*}
	D_{\ell,2}(q) =& \left[ 1- \gamma - \left(1-\frac2\ell \right) \psi\left( \frac1\ell \right) -\frac2\ell \psi\left( \frac2\ell \right) \right] \frac{1}{2\pi \ell z} + \mathcal{O}\left(1\right),
\end{align*}	
and the asymptotic of $D_{\ell,2}(q)$ for $k \geq 2$ and $k|\ell$
\[
	D_{\ell,2}(q) = c_{h,k,1} \left(\frac{k}{2 \pi \ell z}\right)^2 + c_{h,k,2} \left(\frac{k}{2 \pi \ell z}\right)^2 \Log \left( \frac{2\pi \ell z}{k}\right) + \mathcal{O}\left(\frac{k}{z}\right),
\]
where
\[
	c_{h,k,1}=(1-\zeta_k^h)\left[ \ell+ (1+\zeta_k^h) \psi \left( \frac1\ell \right) - \zeta_k^{h}(1+\zeta_k^h) \psi \left( \frac2\ell \right) \right] \quad \text{and}\quad c_{h,k,2}=(1-\zeta_k^h)(1-\zeta_k^{2h}).
\]
Note that $c_{0,1,1}=0$, $c_{0,1,2}=0$, and $c_{h,k,1}\neq 0$ for $k\geq2$.

Similarly, for the case when $k \nmid \ell$, we find that, as $z \to 0$ in $D_\theta$,
\begin{equation}\label{eq:f_F}
	f_{h,k,a}(z) = \frac{\zeta_k^{ah}}{1-\zeta_k^{\ell h}} + \mathcal{O}\left(\frac{z}{k} \right)
	\quad \text{and}\quad
	F_{h,k,a}(z) = \frac{I_{h,k,a}}{z} -  \frac{\zeta_k^{ah}}{2\left(1-\zeta_k^{\ell h}\right)} + \mathcal{O}\left(\frac{z}{k} \right),
\end{equation}
where
\[
	I_{h,k,a} = \int_{0}^\infty f_{h,k,a}(u) du. 
\]
By \eqref{eq:zeta_bound},
\begin{equation}\label{eq:f_const}
	\left| \frac{\zeta_k^{ah}}{1-\zeta_k^{\ell h}} \right| \ll k.
\end{equation}
Next, to estimate $I_{h,k,a}$ for $a>0$, we make a change of variable $e^{-\frac{2\pi \ell u}{k}} \mapsto 1-u$ to get
\begin{equation*}
	I_{h,k,a} = \frac{k}{2\pi \ell} \frac{\zeta_{k}^{ah}}{1-\zeta_k^{\ell h}} \int _0^1  \frac{ \left( 1-u\right)^{\frac{a}\ell - 1}}{1+ \frac{\zeta_{k}^{\ell h}}{1-\zeta_k^{\ell h}} u}  du.
\end{equation*}
If $\zeta_k^{\ell h}=-1$, 
\begin{equation}\label{eq:I_-1}
	I_{h,k,a} = \frac{k \zeta_{k}^{ah}}{2\pi \ell} \int _0^1  \frac{ \left( 1-u\right)^{\frac{a}\ell - 1}}{2-u}  du \ll k \int_{0}^1  (1-u)^{\frac{a}\ell -1} du \ll k.
\end{equation}
For the case when $\zeta_k^{\ell h} \neq -1$, we write $\frac{\zeta_{k}^{\ell h}}{1-\zeta_k^{\ell h}} = R e^{i \theta}$ where $-\pi < \theta \leq \pi$. Note that $\zeta_k^{\ell h} \neq 1$ from $k \nmid \ell$ and $\gcd(h,k)=1$.  Then $ \frac{\pi}{2} + \frac{\pi}{k} \leq |\theta| \leq \pi -\frac{\pi}{2k}$,
from which we have for $0<u<1$ that 
\[
	\left| 1+ \frac{\zeta_{k}^{\ell h}}{1-\zeta_k^{\ell h}} u\right|^2 = 1+ u^2R^2 + 2 u R \cos\left(\theta\right) \geq 1 - \cos^2(\theta) \gg \frac1{k^2}
\]
Thus, we bound $I_{h,k,a}$ for $a>0$ as
\begin{align}\label{eq:I}
	I_{h,k,a} = \frac{k}{2\pi \ell} \frac{\zeta_{k}^{ah}}{1-\zeta_k^{h\ell}} \int _0^1  \frac{ \left( 1-u\right)^{\frac{a}\ell - 1}}{1+ \frac{\zeta_{k}^{\ell h}}{1-\zeta_k^{\ell h}} u}  du
	\ll k^3 \int_{0}^1  (1-u)^{\frac{a}\ell -1} du \ll k^3.
\end{align}
By \eqref{eq:f_F}, \eqref{eq:f_const}, \eqref{eq:I_-1}, and \eqref{eq:I}, we obtain $D_{\ell,2}(q) \ll \frac{k^4}{z}$ if $k \nmid \ell$.
\end{proof}

Similarly, we derive the asymptotics of $D_{\ell,1}(q)$ and $D_{\ell,3}(q)$.
\begin{lemma}\label{lem:D1} 
	Let $\ell \geq 2$ be a fixed integer and $0 \leq \theta < \frac{\pi}{2}$. Then when $k=1$, as $z \to 0$  in $D_\theta$, 
	\begin{align*}
		D_{\ell,1}(q) = \left[ -\gamma - \psi\left(\frac1\ell \right)\right] \frac{1}{2\pi \ell z} + \mathcal{O}\left(1\right).
	\end{align*}
	And we have the bounds of $D_{\ell,1}(q)$ for $k\geq2$ as follows:
	\[
		D_{\ell,1}(q) \ll \frac{k^2}{z^2}  \quad \text{for } k\geq 2 \text{ and } k|\ell, \qquad  D_{\ell,1}(q) \ll \frac{k^3}{z}  \quad \text{for } k\nmid \ell.
	\]	
\end{lemma}

\begin{lemma}\label{lem:D3} 
	Let $\ell \geq 2$ be a fixed integer and $0 \leq \theta < \frac{\pi}{2}$. Then when $k=1$, as $z \to 0$  in $D_\theta$, 
	\begin{align*}
		D_{\ell,3}(q) = \left[\frac32 -\gamma - \left( 1- \frac{3}{2\ell}\right) \left( 1-\frac1\ell \right)\psi\left( \frac1\ell \right) - \frac{1}{\ell} \left( 1-\frac6\ell \right)\psi\left( \frac2\ell \right) -  \frac{3}{2\ell} \left( 1+\frac3\ell \right)\psi\left( \frac3\ell \right) \right] \frac{1}{2\pi \ell z}\\
		  + \mathcal{O}\left(1\right).
	\end{align*}
	And we have the bounds of $D_{\ell,3}(q)$ for $k\geq2$ as follows:
	\[
		D_{\ell,3}(q) \ll \frac{k^4}{z^4}  \quad \text{for } k\geq 2 \text{ and } k|\ell, \qquad  D_{\ell,3}(q) \ll \frac{k^5}{z}  \quad \text{for } k\nmid \ell.
	\]	
\end{lemma}
Using the circle method with Lemmas \ref{lem:xi}, \ref{lem:D2}, \ref{lem:D1}, and  \ref{lem:D3}, we can derive the asymptotics of $d_{\ell,t}(n)$ for $t=1,2,3$, as in the proof of Theorem \ref{thm:b_asymp}. Therefore, we omit the proof of Theorem \ref{thm:d_asymp}.

\section{Inequalities for hook lengths in $\ell$-regular partitions and $\ell$-distinct partitions}

In this section, we give a proof of Theorem \ref{thm:bias_b_d}. In order to prove it, we define functions
\begin{align*}
	g_2(x) &:= \frac32 - \gamma - \psi(x+1) - 2x\big( \psi(2x) - \psi(x) \big),\\
	g_3(x) &:= 2- \gamma -\psi(x+1) +\frac{x}2 (5-3x) \psi(x) -x(1-6x) \psi(2x) -\frac{3x}{2} (1+3x) \psi(3x).
\end{align*}
Then we have the following lemma on $g_2(x)$ and $g_3(x)$.

\begin{lemma}\label{lem:g2_3}
	$g_2(x)$ and $g_3(x)$ are decreasing on the interval $(0,1)$.
\end{lemma}
\begin{proof}
From   \eqref{eq:digamma_series}, we find for $0<x<1$ that
\begin{align*}
	-g'_2(x)&=\psi'(x+1) + 2\left( \psi(2x) - \psi(x) \right)+ 2x\left(2 \psi'(2x) - \psi'(x) \right) \\
	&= \sum_{n \geq 0} \frac1{(n+x+1)^2} + 2 \sum_{n \geq 0} \left(\frac{n}{(n+x)^2}- \frac{n}{(n+2x)^2}\right)   \geq \sum_{n \geq 0} \frac1{(n+2)^2} = \frac{\pi^2}{6}-1 >0.
\end{align*}
Similarly, we observe for $0<x<1$ that
\begin{align*}
	-g'_3(x) 
	&= \psi'(x+1) + \left[ -\frac52 \left(\psi(x)+ x \psi'(x) \right) + \left( \psi(2x) + 2x \psi'(2x) \right) + \frac32 \left( \psi(3x) +3x\psi'(3x) \right) \right]\\
	& \hskip 59pt  + 3x \left( \psi(x) -4 \psi(2x) + 3 \psi(3x) \right) + x^2 \left( \frac32 \psi'(x) -12 \psi'(2x) + \frac{27}2 \psi'(3x) \right)  \\
	&\geq \psi'(x+1)  \geq \frac{\pi^2}{6}-1 >0.
\end{align*}
Hence, $g_2(x)$ and $g_3(x)$ are decreasing on the interval $(0, 1)$. 
\end{proof}

We also provide the limits of $\beta_t(\ell)$ as $\ell \to \infty$ and the inequalities between $\beta_t(\ell)$ for $t=1, 2, 3$. Then Corollary \ref{cor:dis1} follows from Lemma \ref{lem:beta} (2).
\begin{lemma}\label{lem:beta}
	The following are true:
	\begin{enumerate}
		\item For each $t=1, 2, 3$, we have that $\beta_t(\ell) \to 1$ as $\ell \to \infty$.
		\item $\beta_1(\ell) \geq \beta_2(\ell) \geq \beta_3(\ell)$ for $\ell \geq 2$.
	\end{enumerate}
\end{lemma}
\begin{proof}
	(1) By \eqref{eq:digamma_series}, 
	\begin{align}
		\beta_1 \left( \frac1x \right) = - x \sum_{n \geq 0}\left(\frac1{n+1} - \frac1{n+x} \right) &= x \left[ -1 + \frac1x  -  \sum_{n \geq 1}\left(\frac1{n+1} - \frac1{n+x} \right) \right] \nonumber\\
		&= 1 +  x \left[- 1 -  \sum_{n \geq 1}\left(\frac1{n+1} - \frac1{n+x} \right) \right]. \label{eq:beta1}
	\end{align}
	Taking the limit of the last expression as $x \to 0$ will provide $\beta_1(\ell) \to 1$ as $\ell \to \infty$. 
	
	Similarly, we also have that
	\begin{align}
		\beta_2\left(\frac1x\right) &= x\left[ 1- \gamma - \psi( x) + 2x \left( \psi( x) - \psi( 2x )\right) \right] \nonumber\\
		&= x \left[ 1- \left( 1- \frac1x \right) -\sum_{n \geq 1} \left( \frac1{n+1} - \frac{1}{n+x}\right)  - 2x \left( \frac1{x} -\frac1{2x} \right) - 2x \sum_{n \geq 1} \left( \frac1{n+x} - \frac1{n+2x}\right) \right] \nonumber\\
		&= 1+ x \left[ -1  -\sum_{n \geq 1} \left( \frac1{n+1} - \frac{1}{n+x}\right) - 2x \sum_{n \geq 1} \left( \frac1{n+x} - \frac1{n+2x}\right) \right], \label{eq:beta2}
	\end{align}
	and
	\begin{align}
		\beta_3\left(\frac1x\right) 
		&= 1+ x \left[ -1 - \sum_{n \geq 1}\left( \frac1{n+1} - \frac{1}{n+x}\right) - \frac{x}2 \sum_{n \geq 1} \left(  \frac{5-3x}{n+x} - \frac{2(1-6x)}{n+2x} - \frac{3(1+3x)}{n+3x}  \right) \right], \label{eq:beta3}
	\end{align} 
	which yield the limits of $\beta_2(\ell)$ and $\beta_3(\ell)$ as $\ell \to \infty$.
	
	(2) From \eqref{eq:beta1}, \eqref{eq:beta2}, and \eqref{eq:beta3}, we find for $0< x <1$ that
	\[
		\beta_1\left(\frac1x\right)  - \beta_2 \left(\frac1x\right)  = 2x^2 \sum_{n \geq 1} \left( \frac1{n+x} - \frac1{n+2x}\right) > 0.
	\]
	and
	\[
		\beta_2 \left(\frac1x\right)  - \beta_3 \left(\frac1x\right) = \frac{x^2}2 \sum_{n \geq 1} \left(  \frac{(1-3x)}{n+x} + \frac{2(1+6x)}{n+2x} - \frac{3(1+3x)}{n+3x}  \right) > 0. \qedhere
	\]
\end{proof}

Now, we are ready to prove Theorem \ref{thm:bias_b_d}.
\begin{proof}[Proof of Theorem \ref{thm:bias_b_d}]
Let $\alpha_1(\ell):=1-\frac1{\ell}$ and $\alpha_2(\ell)=\alpha_3(\ell):= 1-\frac1{2\ell}$.

(1) First, we consider the case when $t=2$. By Theorems \ref{thm:b_asymp} and \ref{thm:d_asymp}, it is sufficient to prove that
\[
	\alpha_2(\ell) \geq \beta_2(\ell) \quad \text{ for }~ 2 \leq \ell \leq 3 \qquad \text{ and } \qquad \alpha_2(\ell) \leq \beta_2(\ell) \quad \text{ for }~ \ell \geq 4. 
\]
Using $\psi(x+1)=\frac1x + \psi(x)$, we observe that
\begin{align*}
	\ell \big( \beta_2(\ell) - \alpha_2(\ell) \big) &= 1-\gamma - \left( 1 - \frac{2}{\ell} \right)\psi \left( \frac1\ell \right) - \frac{2}{\ell}\psi\left( \frac2\ell \right) - \ell +\frac12 \\
	&= \frac32 -\gamma - \psi \left( 1+\frac1\ell \right) -\frac2\ell \left( \psi \left( \frac2\ell \right) - \psi \left( \frac1\ell \right) \right).
\end{align*}
Then $\beta_2(\ell) - \alpha_2(\ell)=\frac1\ell g_2\left( \frac1\ell \right)$. By Lemma \ref{lem:g2_3}  with $\psi\left( \frac12 \right) = - 2 \log 2 -\gamma$ and $\psi\left( \frac14 \right) = -\frac{\pi}{2} - 3 \log 2 -\gamma$, we obtain for $0<x\leq \frac14$ that
\[
	g_2(x) \geq g_2\left(\frac14\right) = -\frac52 +\frac{\pi}{4} + \frac52\log 2 > 0.
\]
By the reflection formula of the digamma function $\psi(1-x) - \psi(x) = \pi \cot(\pi x)$ and $\psi\left( \frac13 \right) = -\frac{\pi}{2\sqrt{3}} - \frac32 \log 3 -\gamma$,
\[
	g_2\left(\frac13\right) = -\frac32 - \frac{\pi}{6\sqrt{3}} + \frac32\log 2 < 0,
\]
which completes the proof for $t=2$.

For the case when $t=3$,
note that  $\beta_3(\ell) - \alpha_3(\ell)=\frac1\ell g_3\left( \frac1\ell \right)$. By Lemma \ref{lem:g2_3}, we have for $0<x\leq \frac15$ that
\[
	g_3(x) \geq g_3\left( \frac15 \right) > 0.02,
\]
and for $\frac14 \leq x <1$ that
\[
	g_3(x) \leq g_3 \left( \frac14 \right) = -2-\frac{3\pi}{32}+\frac{25}{8} \log 2 < 0.
\]
Hence, the inequalities for $t=3$ are verified.

(2) Setting $r_{\ell,t} := \frac{\beta_t(\ell)}{\alpha_t(\ell)}$, the desired results follow from Theorems \ref{thm:b_asymp}, \ref{thm:d_asymp}, and Lemma \ref{lem:beta} (1).
\end{proof}

\section*{Acknowledgements}
The author would like to thank the referee for helpful comments and suggestions.
The author was supported by the Basic Science Research Program through the National Research Foundation of Korea (NRF) funded by the Ministry of Education (RS-2023-00244423, NRF--2019R1A6A1A11051177).


\end{document}